\documentclass[letterpaper, 10 pt, conference]{ieeeconf}
\IEEEoverridecommandlockouts
\usepackage[centertags]{amsmath}
\usepackage{graphics}
\usepackage{amscd}
\usepackage{amsfonts}
\usepackage{amssymb}
\usepackage{helvet}
\usepackage{rotating}
\usepackage{epsfig}
\usepackage[english]{babel}
\usepackage[latin1]{inputenc}
\usepackage{placeins}
\usepackage{pst-all}
\usepackage{cite}
\newtheorem{thm}{Theorem}[section]
\newtheorem{theorem}[thm]{Theorem}

\newtheorem{definition}[thm]{Definition}

\newtheorem{remark}[thm]{Remark}

\def\R{\mathbb{R}}
\def\N{\mathbb{N}}
\def\U{\mathbb{U}}

\def\arginf{\mathop{\rm arginf}}

\newenvironment{myalign*}{\vspace*{-0.2cm}\align}{\endalign\vspace*{-0.6cm}\\}
\title{\LARGE \bf
Stability of Observer Based Predictive Control for Nonlinear Sampled-data Systems
}
\author{ \parbox{3 in}{\centering J\"{u}rgen Pannek
        \thanks{This work was supported by the DFG priority program 1305 and the Elitenetzwerk Bayern.}\\
        Mathematical Institute\\
        University of Bayreuth\\
        95440 Bayreuth, Germany\\
        {\tt\small juergen.pannek@uni-bayreuth.de}}
        \hspace*{ 0.5 in}
        \parbox{3 in}{ \centering Marcus von Lossow\\
        Mathematical Institute \\
        University of Bayreuth\\
        95440 Bayreuth, Germany\\
        {\tt\small marcus.vonlossow@uni-bayreuth.de}}
}

\begin{document}
\maketitle
\thispagestyle{empty}
\pagestyle{empty}
\begin{abstract}
	We propose a new model predictive control (MPC) approach which is completely based on an observer for the state system. For this, we show semiglobally practically asymptotic stability of the closed--loop for an abstract observer and illustrate our results for a numerical example.
\end{abstract}
\section{Introduction}
Nowadays, a popular control technique for digital control problems is model predictive control (MPC). Within such a method, a prediction of the state trajectory over finite horizon is used for the optimization which is based on estimates of the exact initial state. Then, the first control value is implemented and the problem is shifted forward in time. For these systems we can not expect global stabilization by output feedback, see e.g. \cite{MPD94} for a counterexample. Yet, using suitable controllability and observability assumptions, semiglobal practical asymptotic stabilizability of the system can be guaranteed \cite{ST2003}.\\
In particular the actual state and its estimate do not have to coincide. To avoid the possible divergence of the system Shim and Teel split the available computing time in \cite{ST2003}. First they calculate a good approximation of the actual state using some kind of high gain observer and implement a feedback based on this estimate. The generated delays and errors, however, are handled by robustness of the controller itself and are not integrated in the setup.\\
Here our aim is to incorporate these aspects within the controller design. To this end we consider the by now well know concept of model predictive controllers MPC in a sampled--data fashion. \\
Within the standard setting of MPC with state estimation one utilizes an external approximation of the state and minimizes a given cost functional according to the development of a simulated system with initial values based on the external estimates. This is carried out by appropriately choosing a control sequence for this simulation over a finite horizon. Then the computed control value is implemented at a future time instant and the procedure is repeated, see e.g. \cite{AZ2000,MM95,GMTT2005,JH2005,GNP2007}.\\
Since the internally predicted trajectory is based on estimated and not exact state information, it may deviate from the real one. Even worse, the corresponding control may even be counterproductive considering the desired behavior. In general, this prediction yields a poor estimate of the state at the implementation point, especially if we compare it to the approximation using the (already given) state estimation procedure since such an observer was designed for this task. Hence, this additional knowledge of the system is ignored in this setting instead of utilizing it to construct a comparatively more consistent computation basis. \\
However, simply integrating this approximation into the MPC concept naturally leads to a time acausality, that is future measurements are necessary to generate the estimate, see \cite{RN2004} for details. To overcome this deficiency and at the same time to analytically handle the delays and errors occurring in the output--feedback emulation design of \cite{ST2003} we propose a nonstandard MPC closed--loop scheme. Here the prediction over the control horizon is calculated not by using a model of the plant and some estimated initial value, but by making use of a retarded observer. Hence, the resulting control law is computed according to a simulation of the observer approximating the real (future) state of the plant instead of the forward prediction of a model of the plant which may diverge. Therefore the control is calculated more consistently and an improved behavior of the resulting closed--loop can be expected. Moreover we avoid the necessity of splitting the available computing time between observer and controller, i.e. guaranteeing observability and stabilization, since these parts are merged here within one computing step.\\
The main goal of this paper is to provide a mathematically rigorous stability analysis of the proposed scheme. Section \ref{setupsec} describes the setup and formalize the MPC scheme we propose. In Section \ref{theorem} we show asymptotic stability for our proposed MPC scheme. Finally, we illustrate our setup using an example in Section \ref{example} and draw some conclusions in Section \ref{conclusion}.
\section{Problem formulation}\label{setupsec}
Consider a nonlinear plant
\begin{align}
	\label{Problem formulation:plant}
	\dot{x}(t) = f(x(t), u(t)), \qquad y(t) = h(x(t))
\end{align}
with vector field $f: \R^{n} \times \U \rightarrow \R^n$. Here $x(t) \in \R^{n}$ denotes the state variable, $u \in \U \subset \R^{m}$ the exogenous control variable and $y \in \R^{p}$ the output. Now we want to stabilize the origin based on the structural information available from the plant and the measured outputs using a digital computer with sampling and zero order hold at the sampling time instants $t_{k} = k \cdot T$, $k \in \N$, $T \in \R_{>0}$. \\
In order to achieve that we construct an observer system
\begin{align}
	\label{Problem formulation:observer}
	\dot{\xi}(t) = g(\xi(t), \xi(t - \theta), y(t - \theta), u(t))
\end{align}
with $\theta > 0$ which is based on past information of the output of the plant \eqref{Problem formulation:plant}. Using \eqref{Problem formulation:observer} we calculate a stabilizing sampled--data feedback $u_{T}(\xi(t))$ with zero order hold using a MPC algorithm, see e.g. \cite{GNP2007}. Hence our control law will be based entirely on the observer system and the available measurements.
\begin{figure}[!ht]
	\centering
	\caption{Scheme of the considered control system}
	\label{fig:scheme}
	\psset{xunit=0.75cm,yunit=0.75cm,runit=0.75cm,arrowscale=1}
	\linethickness{0.5pt}
	\begin{pspicture}(-1.5,-3.0)(7.0,1.5)
		\sffamily
		\linethickness{1pt}
		\put(0.0, -0.5){\framebox(2.0, 1.0)}
		\rput[0](1.0, 0.0){\parbox[]{2.0cm}{\centering \small{Observer}}}
		\put(0.25, 0.55){\framebox(1.5, 0.5)}
		\rput[0](1.0, 0.76){\parbox[]{2.0cm}{\centering \small{Buffer}}}
		\psline{->}(2.0, 0.0)(3.5, 0.0)(3.5, -2.0)(2.0, -2.0)
		\put(0.0, -2.5){\framebox(2.0, 1.0)}
		\rput[0](1.0, -2.0){\parbox[]{2.0cm}{\centering \small{Minimizer}}}
		\psline{->}(0.0, -2.0)(-1.5, -2.0)(-1.5, 0.0)(0.0, 0.0)
		\psline{->}(1.0, -2.5)(1.0, -3.0)(5.0, -3.0)(5.0, -1.5)
		\put(4.0, -1.5){\framebox(2.0,1.0)}
		\rput[0](5.0, -1.0){\parbox[]{2.0cm}{\centering \small{Plant}}}
		\psline{->}(5.0, -0.5)(5.0, 0.75)(1.75, 0.75)
	\end{pspicture}
\end{figure}
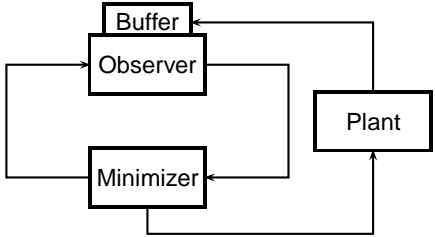
\FloatBarrier
Next we use the constructed feedback to control the plant \eqref{Problem formulation:plant}, and hence the sampled--data closed--loop system is given by
\begin{align*}
	\dot{x}(t) = f(x(t), u_{T}(\xi(t_{j}))), \quad t \in [t_{j}, t_{j + 1}).
\end{align*}
Our goal is to show that we can modify the usual model predictive control algorithm to an observer based predictive control algorithm (OBPC) which semiglobally stabilizes the system \eqref{Problem formulation:plant}.\\
To this end we compute the observer trajectory according to
\begin{align}
	\label{Problem formulation:OBPC:observer} 
	\dot{\xi}(t) = g(\xi(t), \xi(t - N T), y(t - N T), v_{j})
\end{align}
for $t \in [t_{j}, t_{j + 1})$, $j \in \{ 0, \ldots, N - 1 \}$. The solution of the system \eqref{Problem formulation:OBPC:observer} at time $t$ will be denoted by $\xi(t, \xi_0(\cdot), v)$ or $\xi(t)$ given the initial values $\xi(\theta)$, $\theta \in [t_0 - N T, t_0]$, and control sequence $v_{[0, \ldots, N-1]} \in \U^N$. Also we will assume that a unique solution of the observer system $g(\cdot, \cdot, \cdot, \cdot): \R^{n} \times \R^{n} \times \R^{p} \times \R^{m} \rightarrow \R^{n}$ exists for any initial value in a given compact neighborhood $\mathcal{N}$ containing the origin in the first and second argument. Additionally we assume that $y(\tau)$, $\tau \in [t_{0} - N T, t_{0}]$, are known output values of the plant.\\
Based on system \eqref{Problem formulation:OBPC:observer} we calculate an open--loop optimal control of zero order hold
\begin{align}
	\label{Problem formulation:OBPC:control}
	\hat{u}_{[0, N - 1]} = \arginf_{v_{[0, N - 1]}} J_{N} (\xi(t_0), v_{[0, N-1]})
\end{align}
such that the cost functional
\begin{align*}
	J_{N}(\xi(t_0), v_{[0, N - 1]}) := \sum_{j = 0}^{N - 1} \int\limits_{t_{j}}^{t_{j + 1}} l(\xi(t), v_{j}) dt \, + \, F(\xi(t_{N}))
\end{align*}
is minimized. Similar to standard MPC the length of the horizon $[t_{0}, t_{0} + N T]$ is fixed by the sampling time $T \in \R_{>0}$ and the parameter $N \in \N$. Analogously we call $l: \R^{n} \times \R^{m} \rightarrow \R_{\geq 0}$ stage cost and $F: \R^{n} \rightarrow \R_{>0}$ terminal cost.\\
Now we implement the first element $\hat{u}_0$ of the sequence $\hat{u}_{[0, N - 1]}$ into the state system \eqref{Problem formulation:plant}. The resulting output values $y(\tau)$, $\tau \in [t_{0}, t_{0} + T]$ are measured and stored. Then we can continue with the usual MPC procedure, shift the horizon of the open--loop optimization in a receding horizon fashion and iterate these steps. The resulting control will be denoted by
\begin{align}
	\label{Problem formulation:OBPC:control2}
	u_T(\xi(t_j)) := \hat{u}_0^{(j)}
\end{align}
where $\hat{u}_0^{(j)}$ represents the first element of the control sequence of the $j$-th iterate.\\
For reasons of simplicity we assume that the infimum in \eqref{Problem formulation:OBPC:control} is attained. Additionally we assume $f(x, u)$ to be locally Lipschitz in $x$ and $h: \R^{n} \rightarrow \R^{p}$ to be locally Lipschitz with $h(0) = 0$.
\begin{remark}
	Within the standard MPC setting, the feedback is computed based on an \textit{estimate of the state}, see e.g. \cite{MM95,MNS2001}. To obtain a minimizing control, however, a \textit{model of the plant} is simulated forward in time.
	\begin{figure}[!ht]
		\centering
		\caption{Standard scheme for MPC}
		\label{fig:old scheme}
		\psset{xunit=0.75cm,yunit=0.75cm,runit=0.75cm,arrowscale=1}
		\linethickness{0.5pt}
		\begin{pspicture}(-1.5,-3.0)(7.0,1.5)
			\sffamily
			\linethickness{1pt}
			\put(0.0, -0.5){\framebox(2.0, 1.0)}
			\rput[0](1.0, 0.0){\parbox[]{2.0cm}{\centering \small{Simulation}}}
			\psline{->}(2.0, 0.0)(3.5, 0.0)(3.5, -2.0)(2.0, -2.0)
			\put(0.0, -2.5){\framebox(2.0, 1.0)}
			\rput[0](1.0, -2.0){\parbox[]{2.0cm}{\centering \small{Minimizer}}}
			\psline{->}(0.0, -2.0)(-1.5, -2.0)(-1.5, 0.0)(0.0, 0.0)
			\psline{->}(1.0, -2.5)(1.0, -3.0)(5.0, -3.0)(5.0, -2.5)
			\put(4.0, -2.5){\framebox(2.0,1.0)}
			\rput[0](5.0, -2.0){\parbox[]{2.0cm}{\centering \small{Plant}}}
			\psline{->}(5.0, -1.5)(5.0, -0.5)
			\put(4.0, -0.5){\framebox(2.0,1.0)}
			\rput[0](5.0, 0.0){\parbox[]{2.0cm}{\centering \small{Observer}}}
			\psline{->}(5.0, 0.5)(5.0, 1.0)(1.0, 1.0)(1.0, 0.5)
		\end{pspicture}
	\end{figure}
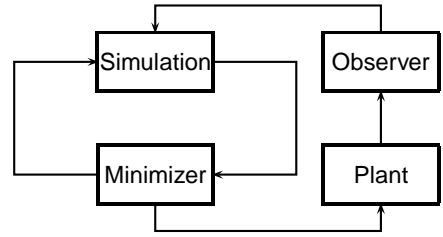
	Since there are no restraints on the plant, the initial error may --- in the worst case --- be amplified by the factor $e^{L\cdot NT}$ where $L$ denote the local Lipschitz constant of the plant and $NT$ the simulated horizon. Hence, a control computed on this basis is not computed consistently and the resulting closed--loop behavior may be far from optimal.\\
	Comparing Figures \ref{fig:scheme} and \ref{fig:old scheme} reveals the difference between the proposed and the standard setting of MPC: Here, we do not simulate the plant itself but use the observer for the prediction within the MPC algorithm. The feedback is therefore not depending on the state but on an observer of the state which is time--retarded. Hence, if the retarded observer approximates the state of the plant, the control is computed based on a simulation converging towards the actual state of the plant instead of a possibly diverging one. For this reason we expect an enhanced behavior of the closed--loop.
\end{remark}
\begin{remark}\label{emulateremark}
	In practice we cannot measure and store the output values for all $\tau \in [t_{0}, t_{0} + T]$. For practical applications one can overcome this problem as follows:\\
	i) Introduce a second (output) sampling time $\hat{T} \ll T$ such that only $y(\tau)$, $\tau = t_{0} + i \hat{T}$, $i = 0, \ldots, M$ with $M \hat{T} < T \leq (M + 1) \hat{T}$ are stored. However, if this sampling time is sufficiently small we can use results from \cite{NTK99} to guarantee that, given certain consistency and stability conditions, a control $u$ which stabilizes
	\begin{align}
		\label{Problem formulation:Remark:observer discrete}
		\dot{\xi}(t) = g(\xi(t), \xi(t_{j} - N T), y(t_{j} - N T), u(t))
	\end{align}
	for $t \in [t_{j}, t_{j + 1})$, $t_{j} = t_{0} + j \hat{T}$ also stabilizes \eqref{Problem formulation:OBPC:observer}. \\
	Conversely, it was shown in \cite{LNT2002} that if the system \eqref{Problem formulation:OBPC:observer} is stabilized by $u(t)$ for a continuous $y$ and if the sampling time is sufficiently small then also the in $y$ ``emulated system" \eqref{Problem formulation:Remark:observer discrete} is stabilized.\\
	ii) Alternatively one can use an interpolation procedure to approximate the past continuous output $y$ based on measurements at the sampling instants $t_j$. However, if the sampling time is large then one still needs to introduce a finer output grid to obtain reliable information.
\end{remark}
In order to define stability of the closed--loop system, i.e. the plant in Figure \ref{fig:scheme} given by
\begin{align}
	\label{Problem formulation:closed--loop}
	\dot{x}(t) & = f(x(t), u_{T}(\xi(t_{j}))), \qquad y(t) = h(x(t)) \\
	\dot{\xi}(t) & = g(\xi(t), \xi(t - N T), y(t - N T), u_{T}(\xi(t_{j}))) \nonumber
\end{align}
for $t \in [t_{j}, t_{j + 1})$ where $u_T$ is given by \eqref{Problem formulation:OBPC:control2}, we introduce the notion of comparison functions: A function $\gamma: \R_{\geq 0} \rightarrow \R_{\geq 0}$ is of class $\mathcal{K}$ if it is continuous, zero at zero and strictly increasing. It is of class $\mathcal{K}_{\infty}$ if it is also unbounded. A function belongs to class $\mathcal{L}$ if it is strictly positive and it is decreasing to zero as its argument tends to infinity. Last, a function $\beta: \R_{\geq 0} \times \R_{\geq 0} \rightarrow \R_{\geq 0}$ is called $\mathcal{KL}$--function if for every fixed $t \geq 0$ the function $\beta(\cdot, t)$ is an element of class $\mathcal{K}$ and for each fixed $s > 0$ the function $\beta(s, \cdot)$ is of class $\mathcal{L}$. \\
Here we seek stability of the closed--loop in the following sense:
\begin{definition}[Semi--global practical Stability]\label{Problem formulation:stability}$\\ $
	Consider a control system
	\begin{align*}
		\dot{x}(t) & = F(x(t), x(t - NT), u_T)
	\end{align*}
	where $F(0, 0, 0) = 0$ and a family of control functions $u_T$ for $T \in (0, T^\star]$. If there exists a function $\beta \in \mathcal{KL}$ and a pair $( \Delta_{1}, \Delta_{2} )$ of positiv real--valued numbers such that the inequality
	\begin{align*}
		\| x(t, x_0, u_T) \| \leq \max \left\{ \beta(\| x_0 \|, t), \Delta_{2} \right\}
	\end{align*}
	holds for all $t \geq t_0$ and all initial values $x_0(\theta) \in B_{\Delta_{1}}^{0}(0)$, $\theta \in [t_0 - N T, t_0]$, then the origin is called semi--globally practically asymptotically stabilizable.
\end{definition}
This definition includes the standard definition of semi--globally practically asymptotically stability if we consider $F$ to be independent of past information. Additionally we can use it in the context of the retarded observer system \eqref{Problem formulation:OBPC:observer} and the combined system \eqref{Problem formulation:closed--loop}.
\section{Stability of the closed--loop}\label{theorem}
We now show the stability of our proposed scheme. Note that if we applied an observer which is not retarded we would require the measured output $y$ at times $t \in [t_0, t_0 + N T]$, see \cite{GK2001} for methodology and \cite{PA2004,AN2004,AP2006} for necessary/sufficient existence conditions of observer based output--feedbacks. Since this represents a time interval in the future we lack these values. In order to overcome this acausality we do not follow the way described in \cite{RN2004} and use the special structure of the ENOCF observer but retard our observer, see \cite{YS98} for an observer construction in the linear case. To show stability of our scheme we consider an abstract observer, i.e. with no special structure. Here we will only assume that the observer possesses the following properties:
\begin{itemize}
	\item[] {\bf (A1)} If $\xi_0(\theta) = x(\theta)$, $\theta \in [- NT, 0]$, $T \geq 0$, $N \in \N$ holds for the history of $x$ and $\xi$, then $\xi(t, \xi_0(\cdot), u_T) \equiv x(t, x_0, u_T)$ holds for all $t \geq 0$ and all $u_T \in \U^\N$.
	\item[] {\bf (A2)} If $\xi(0) \not = x(0)$ there exists a function $\tilde{\beta} \in \mathcal{KL}$ such that
	\begin{align*}
		\| x(t, x_0, u_T) - \xi(t, \xi_0(\cdot), u_T) \| \leq \tilde{\beta}( \| x(0) - \xi(0) \|, t)
	\end{align*}
	holds for the estimation error for all $u_T \in \U^\N$.
\end{itemize}
Based on this observer we generate a feedback which will give us the stability of the interconnected system.
\begin{theorem}[Observer--based Feedback]\label{Theorem: Observer based Feedback} $\\ $
	Given constants $\nu \geq (1 + \alpha) \Delta_{1}$, $\alpha > 0$, we consider the combined system \eqref{Problem formulation:closed--loop}
	\begin{align}
		\label{Theorem:plant}
		\dot{x}(t) = f(x(t), u_T(t)), \qquad y(t) = h(x(t))
	\end{align}
	and its associated observer system
	\begin{align}
		\label{Theorem:observer}
		\dot{\xi}(t) = g(\xi(t), \xi(t - N T), y(t - N T), u_T(t))
	\end{align}
	\eqref{Theorem:observer} satisfying the previous assumptions. Additionally initial values $\xi(\theta)$ and $x(\theta)$ are given which satisfy $\| x(\theta) - \xi(\theta) \| \leq \nu$ for all $\theta \in [t_{0} - N T, t_{0}]$. Moreover we assume that \eqref{Theorem:observer} is $(\Delta_{1}, \Delta_{2})$--semiglobally practically asymptotically stabilized by the feedback control law
	\begin{align}
		\label{Theorem:observer:feedback}
		u_T(t) := \mu(\xi(t), \xi(t - N T), y(t - N T)).
	\end{align}
	for all possible outputs $y \in \R^p$. Then the combined system \eqref{Theorem:plant}, \eqref{Theorem:observer} is $(\overline{\Delta}_{1}, \overline{\Delta}_{2})$--semiglobally practically asymptotically stabilized by the feedback law \eqref{Theorem:observer:feedback} with $\overline{\Delta}_1 := \nu - \Delta_{1}$ and $\overline{\Delta}_2 := 4 \Delta_{2}$.
\end{theorem}
\begin{proof}
	Since \eqref{Theorem:observer} is $(\Delta_{1}, \Delta_{2})$--semiglobally practically asymptotically stable for the chosen control there exists an attraction rate $\beta \in \mathcal{KL}$ and a pair $(\Delta_{1}, \Delta_{2})$ of positive real--valued numbers such that
	\begin{align*}
		\| \xi(t) \| \leq \max \left\{ \beta( \| \xi(t_{0}) \|, t), \Delta_{2} \right\}
	\end{align*}
	for all $t \geq t_0$ and $\xi(\theta) \in B_{\Delta_{1}}^{0}(0)$, $\theta \in [t_0 - NT, t_0]$. Using the triangle inequality and (A2) we obtain
	\begin{align}
		\label{thm:observer:proof:eq1}
		\| x(t) \| & \leq \| \xi(t) \| + \| x(t) - \xi(t) \| \\
		& \hspace*{-0.3cm} \leq \max \left\{ \beta( \| \xi(t_{0}) \|, t), \Delta_{2} \right\} + \tilde{\beta}( \| x(t_{0}) - \xi(t_{0}) \|, t). \nonumber
	\end{align}
	If we consider a mixed $1$--norm of this vector
	\begin{align*}
		\rho(t) := \left\| \left( \begin{array}{c} \| x(t) \| \\ \| \xi(t) \| \end{array} \right) \right\|_{1}
	\end{align*}
	as a function of time we get
	\begin{align}
		\rho(t) \leq & \max \{ \beta( \| \xi(t_{0}) \| , t), \Delta_{2} \} + \tilde{\beta}( \| x(t_{0}) - \xi(t_{0}) \|, t ) \nonumber \\
		& + \max \{ \beta( \| \xi(t_{0}) \|, t), \Delta_{2} \} \nonumber \\
		& \hspace*{-0.5cm}\leq 2 \max \left\{ \beta \left( 
		\rho(t_0), t \right), \Delta_{2} \right\} 
		+ \tilde{\beta}( \| x(t_{0}) \| + \| \xi(t_{0}) \|, t ) \nonumber \\
		& \hspace*{-0.5cm}= 2 \max \left\{ \beta \left( 
		\rho(t_0), t \right), \Delta_{2} \right\} 
		+ \tilde{\beta}( \rho(t_0), t ). \label{thm:observer:proof:eq2}
	\end{align}
	We can conclude by $\| x(t_{0}) - \xi(t_{0}) \| \leq \nu$, $\nu \geq (1 + \alpha) \Delta_{1}$, $\alpha > 0$, and $\xi(\theta) \in B_{\Delta_{1}}^{0}(0)$, $\theta \in [t_0 - NT, t_0]$, that these conditions can only be satisfied for all $x(\theta) \in B_{\nu - \Delta_{1}}(0)$, $\theta \in [t_0 - NT, t_0]$. Since $\alpha > 0$ this is a nonempty closed set and we can restrict ourselves to the set $B_{\nu - \Delta_{1}}^{0}(0)$, in particular choose
	\begin{align}
		\overline{\Delta}_{1} := \nu - \Delta_{1}
	\end{align}
	and conclude $x(t_{0}) \in B_{\overline{\Delta}_{1}}^{0}(0) \subset B_{\nu}^{0}(0)$.\\
	If $\rho(t_0) = 0$ we are done since $x(t_{0}) = \xi(t_{0}) = 0$ and due to the $\mathcal{KL}$--property of $\tilde{\beta}$ the error will be zero for all $t \geq t_{0}$. Hence we can identify the trajectories of $\xi(t) \equiv x(t)$ for all $t \geq t_{0}$, set
	\begin{align}
		\overline{\Delta}_{2} := \Delta_{2}
	\end{align}
	and get
	\begin{align*}
		\| x(t) \| \leq \max \left\{ \beta( \| x(t_{0}) \|, t), \overline{\Delta}_{2} \right\}
	\end{align*}
	for all $t \geq t_{0}$ and all $x(t_{0}) \in B_{\overline{\Delta}_{1}}^{0}(0)$. \\
	If $\rho(t_0) \not = 0$ we want to obtain
	\begin{align*}
		\rho(t) \leq \max \left\{ \overline{\beta} \left( 
		\rho(t_0), t \right), \overline{\Delta}_{2} \right\}
	\end{align*}
	To this end we consider \eqref{thm:observer:proof:eq2} for any fixed $t \geq t_{0}$. Then one of the following cases holds:
	\begin{itemize}
		\item[(1)] $\beta \left( \rho(t_0), t \right) \geq \Delta_{2}$: Using this in \eqref{thm:observer:proof:eq2} it follows that
		\begin{align*}
			\rho(t) \leq 2 \beta \left( \rho(t_0), t \right) + \tilde{\beta} \left( \rho(t_0), t \right)
		\end{align*}
		and we can directly define
		\begin{align*}
			\overline{\beta}^{(1)}(r, t) := 2 \beta(r, t) + \tilde{\beta}(r, t).
		\end{align*}
		\item[(2)] $\beta \left( \rho(t_0), t \right) < \Delta_{2}$: From \eqref{thm:observer:proof:eq2} we can see
		\begin{align*}
			\rho(t) \leq 2 \Delta_{2} + \tilde{\beta} \left( \rho(t_0), t \right)
		\end{align*}
		Now we consider the following cases:
		\begin{itemize}
			\item[(2a)] $\tilde{\beta} \left(\rho(t_0), t \right) < 2 \Delta_{2}$: Here we get $\rho(t) < 4 \Delta_{2}$ and we can use
			\begin{align*}
				\overline{\Delta}_{2}^{(2)} := 4 \Delta_{2}.
			\end{align*}
			\item[(2b)] $\tilde{\beta} \left( \rho(t_0), t \right) \geq 2 \Delta_{2}$: In this case we get $\rho(t) \leq 2 \tilde{\beta} \left( \rho(t_0), t \right)$ and define
			\begin{align*}
				\overline{\beta}^{(2)}(r, t) := 2 \tilde{\beta}(r, t).
			\end{align*}
		\end{itemize}
	\end{itemize}
	Using these results we see that
	\begin{align*}
		\rho(t) \leq & \max \left\{ \overline{\beta}^{(1)} \left( 
		\rho(t_0) , t \right), \overline{\beta}^{(2)} \left(
		\rho(t_0), t \right), \overline{\Delta}_{2}^{(2)} \right\}.
	\end{align*}
	Now we can simplify this inequality since the maximum of two $\mathcal{KL}$--functions is still a $\mathcal{KL}$--function. Hence the stated inequality holds for all $t \geq t_{0}$ if we use
	\begin{align*}
		\overline{\beta}( r , t) & := \max \left\{ \overline{\beta}^{(1)}(r, t), \overline{\beta}^{(2)}(r, t) \right\}, \quad \overline{\Delta}_{2} & := 4 \Delta_{2}.
	\end{align*}
	Also there exist $T$, $\hat{T} \geq t_{0}$ such that
	\begin{align*}
		\overline{\beta} \left( 
		\rho(t_0), t \right) &
		\begin{cases}
			\geq \overline{\Delta}_{2} \;, \quad t_{0} \leq t < T \\
			< \overline{\Delta}_{2} \;, \quad t \geq T
		\end{cases} \quad \text{and} \\
		\beta \left( \| \xi(t_{0}) \| , t \right) &
		\begin{cases}
			\geq \Delta_{2} \;, \quad t_{0} \leq t < \hat{T} \\
			< \Delta_{2} \;, \quad t \geq \hat{T}
		\end{cases}.
	\end{align*}
	Hence the trajectory of the $x$--system will stay within $B_{\overline{\Delta}_{2}}^{0}(0)$ for $t \geq \max \{ T, \hat{T} \}$.
\end{proof}
\begin{remark}
	The assumptions regarding the observer system \eqref{Theorem:observer} in Theorem \ref{Theorem: Observer based Feedback} are stronger then ISS. However, we suppose that a generalization to ISS is possible.
\end{remark}
\begin{remark}
	The stated bounds can be tightened at a certain cost. If one considers \eqref{thm:observer:proof:eq1} and the estimate $\| \xi(t_{0}) \| = \| x(t_{0}) \| + \nu$ we get $\| x(t) \| \leq \tilde{\beta}( 2 \| x(t_{0}) \| + \nu,$ $t) + \max \left\{ \beta( 2 \| x(t_{0}) \| + \nu, t), \Delta_{2} \right\}$. Now we have to consider the cases
	\begin{align*}
		\text{(i)} \; \| x(t_{0}) \| \geq \varepsilon \qquad \text{(ii)} \; \| x(t_{0}) \| < \varepsilon
	\end{align*}
	where $0 < \varepsilon \leq \nu$. In case (i) there exists a $\lambda \in \R$: $\nu = \lambda \varepsilon \leq \lambda \| x(t_{0}) \|$ and it follows that
	\begin{align*}
		\| x(t) \| \leq \max \left\{ \overline{\beta}( \| x(t_{0}) \|, t), \overline{\Delta}_{2} \right\}
	\end{align*}
	where $\overline{\beta}( r , t) := \beta( (2 + \lambda) \mu r , t) + \tilde{\beta}( (2 + \lambda) r , t)$ and $\overline{\Delta}_{2} := \Delta_{2} + \overline{\varepsilon}$. Here $\mu \geq 1$ needs to be chosen as follows: Set $\overline{\varepsilon} > 0$ arbitrary. Hence there exists a $T \geq t_{0}$ such that
	\begin{align*}
		\tilde{\beta}( (2 + \lambda) \| x(t_{0}) \| , t) \leq \tilde{\beta}( (2 + \lambda) \overline{\Delta}_{1} , t)
		\begin{cases}
			> \overline{\varepsilon}, \; t_{0} \leq t < T \\
			\leq \overline{\varepsilon}, \; t \geq T.
		\end{cases}
	\end{align*}
	and moreover a $\mu \in \R$ such that $\overline{\beta}( \varepsilon , t) > \Delta_{2} + \tilde{\beta}( (2 + \lambda) \overline{\Delta}_{1} , t)$ for all $t_{0} \leq t < T$. Hence there exists a $\hat{T} \geq T$ such that
	\begin{align*}
		\overline{\beta}( \| x(t_{0}) \| , t) 
		\begin{cases}
		> \overline{\Delta}_{2} \;, \quad t_{0} \leq t < \hat{T} \\
		\leq \overline{\Delta}_{2} \;, \quad t \geq \hat{T}
		\end{cases}
	\end{align*}
	for all $\| x(t_{0}) \| \geq \varepsilon$ where $\overline{\beta}$ is a $\mathcal{KL}$--function and the stated stability conditions hold since we have only enlarged the right hand side.\\
	The drawback here is that in case (ii) we definitely need $\overline{\Delta}_{2} \geq \overline{\beta}\left( \varepsilon, 0 \right)$ to guarantee semiglobally practically asymptotic stability. To see this consider $\| x(t_{0}) \| = 0$. Then the invariance of the set $\{ 0 \}$ is violated and hence no $\mathcal{KL}$--function can be found. Hence semiglobally practically asymptotic stability can only be guaranteed if the maximum of the right hand side is always given by $\overline{\Delta}_{2}$.\\
	Making use of the $\mathcal{L}$--property of the function $\overline{\beta}$ we can conclude that $\overline{\Delta}_{2} := \overline{\beta}( \varepsilon, 0) + \Delta_{2}$ and hence the trajectory of the $x$--system will stay within $B_{\overline{\Delta}_{2}}^{0}(0)$.
\end{remark}
\begin{remark}
	Note that by increasing $\alpha$ the basin of attraction $B_{\overline{\Delta}_{1}}^{0}(0)$ can be enlarged as long as the other conditions are still valid.
\end{remark}
\begin{remark}
	One has to keep in mind that there exists a fixed delay between an occurring disturbance in the $x$--system at time $t_{1} \geq t_0 - N T$ and the reaction of the observer $\xi$ at time $t_{2} = t_{1} + N T$. Since the length of this delay is exactly equal to the length of the OBPC horizon this disturbance will still be taken into account.
\end{remark}
\section{Example}\label{example}
In order to check our proposed scheme we consider two simple examples
\begin{align}
	\label{examples}
	\dot{x}(t) & = A_i x(t) + B u(t) \\
	y(t) & = C x(t) \nonumber
\end{align}
where
\begin{align*}
 	A_1 & = \begin{pmatrix} -1 & 1 \\ 1 & -1 \end{pmatrix}, \; A_2 = \begin{pmatrix} 0 & 1 \\ -1 & 0 \end{pmatrix}, \\
	B & = \begin{pmatrix} 1 & 0 \\ 0 & 1 \end{pmatrix} \; \text{and} \; C = \begin{pmatrix} 1 \\ 0 \end{pmatrix}^\top.
\end{align*}
For these examples we constructed the Luenberger observers
\begin{align*}
	\dot{\xi}(t) = A_i \xi(t) + B u(t) - \Lambda(\lambda_i) K_i ( C \xi(t) - y(t) )
\end{align*}
with
\begin{align*}
	\Lambda(\lambda_i) = \begin{pmatrix} \lambda_i & 0 \\ 0 & \lambda_i^2 \end{pmatrix} \; \text{and} \; K_i = \begin{pmatrix} k_{i,1} \\ k_{i,2} \end{pmatrix}.
\end{align*}
Using the parameter $\lambda_1 = 1.2$, $k_{i,1} = 1$ and $k_{i,2} = 0.5$ the eigenvalues of $(A - \Lambda(\lambda_1) K_i C)$ are negative and hence the observer converges.\\
As a retarded observer we used the same construction with a time shift of the prediction horizon length $\tau = NT$ in the measurement component, that is
\begin{align*}
	\dot{\xi}(t) = A_i \xi(t) + B u(t) - \Lambda(\lambda_i) K_i ( C \xi(t - \tau) - y(t - \tau) ).
\end{align*}
To show convergence we consider the Lyapunov candidate
\begin{align*}
	V(\eta(t)) = \frac{1}{2} \eta(t)^\top \Lambda^{-1}(\lambda) P_i \Lambda^{-1}(\lambda) \eta(t)
\end{align*}
for the error $\eta(t) = \xi(t) - x(t)$ with $P_i = P_i^\top > 0$ such that
\begin{align*}
	P_i \big( A_i - \Lambda(\lambda) K_i C \big) + \big( A_i - \Lambda(\lambda) K_i C \big)^\top P_i = - \text{Id},
\end{align*}
holds. The error dynamic is given by
\begin{align*}
	\dot{\eta}(t) = A_i \eta(t) - \Lambda(\lambda) K_i C \eta(t - \tau).
\end{align*}
which gives us
\begin{align*}
	\dot{V}(\eta(t)) = & \eta(t)^\top W(\lambda) \big( A_i \eta(t) - \Lambda(\lambda) K_i C \eta(t - N T) \big).
\end{align*}
where $W_i(\lambda) := \Lambda^{-1}(\lambda) P_i \Lambda^{-1}(\lambda)$. Using the integral equation for $\eta(t - NT)$ we obtain
\begin{align*}
	\eta(t - NT) & = \eta(t) + \int\limits_{t}^{t - NT} \dot{\eta}(\tau) d\tau \\
	& \hspace*{-1cm} = (\text{Id} + e^{-A_i \cdot NT}) \eta(t) + e^{\Lambda(\lambda) K_i C \cdot NT} \eta(t - NT)
\end{align*}
The parameter $\lambda$, $k_1$ and $k_2$ were chosen such that $(\text{Id} - e^{\Lambda(\lambda) K_i C \cdot NT})$ is invertible, hence we obtain
\begin{align*}
	\dot{V}(\eta(t)) & = \eta(t)^\top W(\lambda) \big( A_i - \Lambda(\lambda) K_i C \\
	& \quad (\text{Id} - e^{\Lambda(\lambda) K_i C \cdot NT})^{-1} (\text{Id} + e^{-A_i \cdot NT}) \big) \eta(t).
\end{align*}
Since $W(\lambda)$ is positive definite it remains to show that
\begin{align*}
	\mathcal{A}_i := & \big( A_i - \Lambda(\lambda) K_i C (\text{Id} - e^{\Lambda(\lambda) K_i C \cdot NT})^{-1} \\
	& \qquad (\text{Id} + e^{-A_i \cdot NT}) \big)
\end{align*}
has negative eigenvalues which can be verified for both examples given the previously mentioned parameter.\\
Moreover the value function is bounded from above since using the properties of $\Lambda$ we have $V(\eta(t)) \leq \frac{\lambda_{\max}(P)}{2 \lambda^{2}} \| \eta(t) \|^{2}  =: \alpha_2( \| \eta(t) \| )$ and similarly from below due to $V(\eta(t)) \geq \frac{\lambda_{\min}(P)}{2 \lambda^{2 n}} \| \eta(t) \|^{2} =: \alpha_1( \| \eta(t) \| )$. Figure \ref{Figure:MPC:ex1} shows the resulting trajectories for the standard MPC scheme using the initial values $x_0 = (11, 8)$ and $\xi_0 = (0, 0)$. The sampling time is set to $T = 0.1$ and the horizon length parameter is taken to be $N = 5$
\begin{figure}[!ht]
	\caption{Resulting trajectories for Example 1 using standard MPC}
	\label{Figure:MPC:ex1}
	\includegraphics[width=0.48\textwidth]{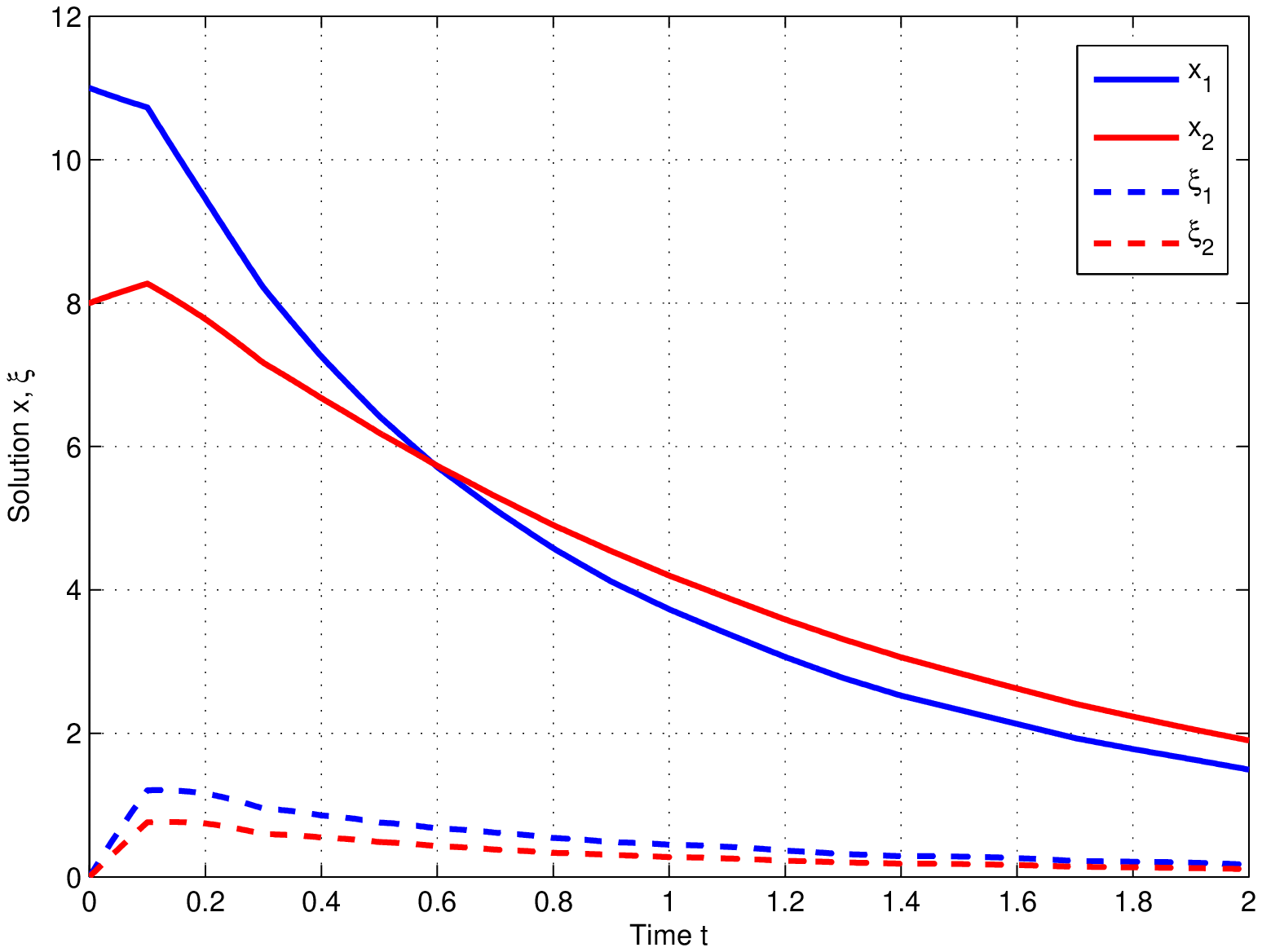}
\end{figure}
Here no optimization takes place within the first step of the MPC algorithm since the initial guess of the observer is set to zero. Note that the resulting control may even exhibit a wrong sign according to the observed state. The observer based predictive controller on the other hand recognizes this deviation, cf. Figure \ref{Figure:OBPC:ex1}. Here we chose the initialization of the history of the system under control to be constant at $(11, 8)$.
\begin{figure}[!ht]
	\caption{Resulting trajectories for Example 1 using OBPC}
	\label{Figure:OBPC:ex1}
	\includegraphics[width=0.48\textwidth]{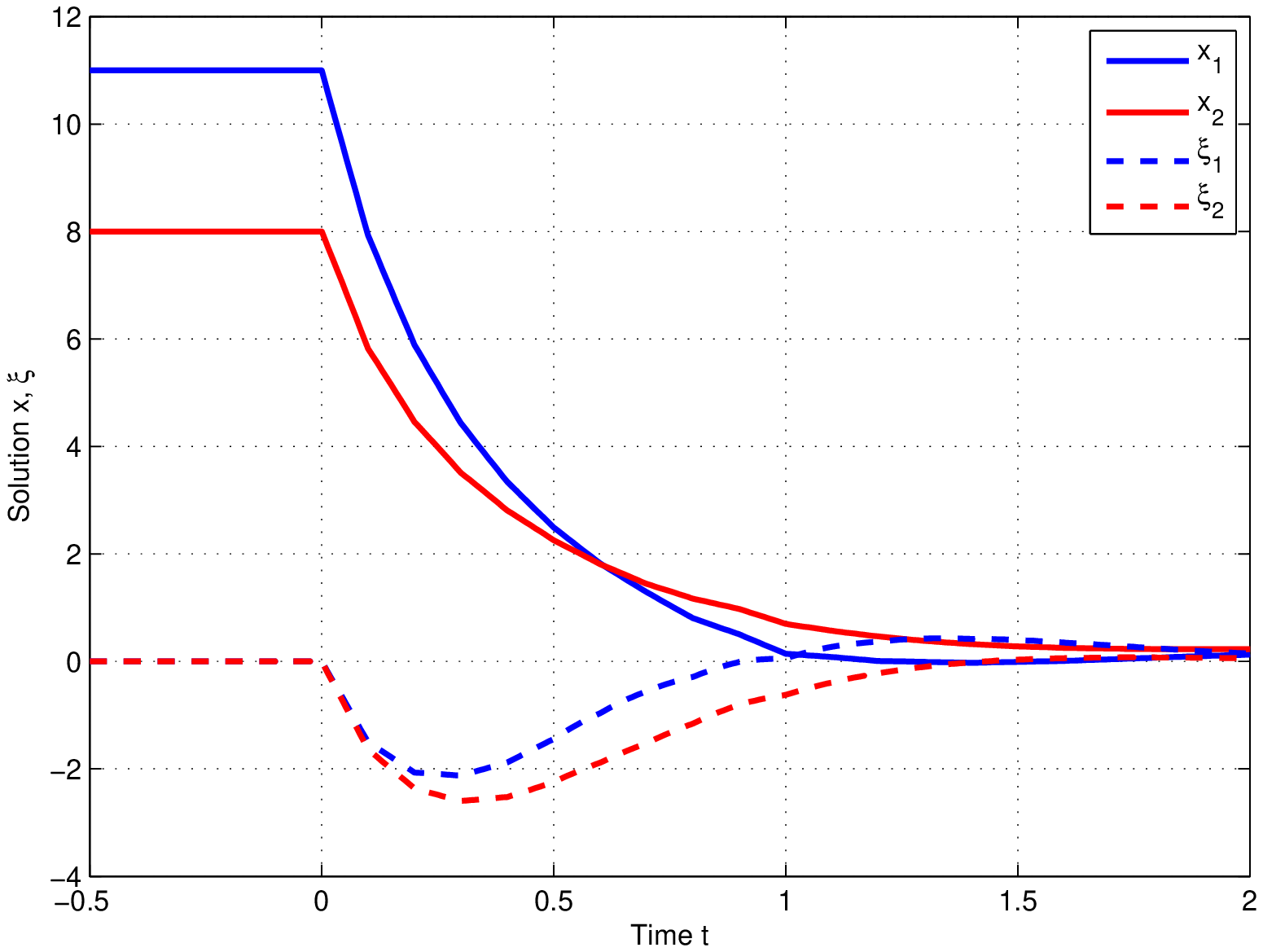}
\end{figure}
Moreover the transient behavior of the observer steers the system state even stronger to zero since it points in negative direction first. Here one has to keep in mind that observer and system state changed places, hence the system state converges towards the observer state. Considering a variety of initial values the observer based variant seems to enhance the stabilization process for this example.\\
Within our second example we do neither change any of the parameter nor the initial values. At first glance of Figures \ref{Figure:MPC:ex2} and \ref{Figure:OBPC:ex2}, one could suspect that the presented setup does not work properly for this example. Analyzing the background we see that the resulting swinging trajectory is due to the poor convergence rate of the retarded observer coming from our construction and is not due to the governing control setup. Particularly in contrast to our first example, the eigenvalues of the Lyapunov function matrix $(W(\lambda) \cdot \mathcal{A})$ of the error development cannot be tuned arbitrarily. The Luenberger observer considered in the standard MPC setup, however, does not exhibit this difficulty and is advantageous for this reason.
\begin{figure}[!ht]
	\caption{Resulting trajectories for Example 2 using standard MPC}
	\label{Figure:MPC:ex2}
	\includegraphics[width=0.48\textwidth]{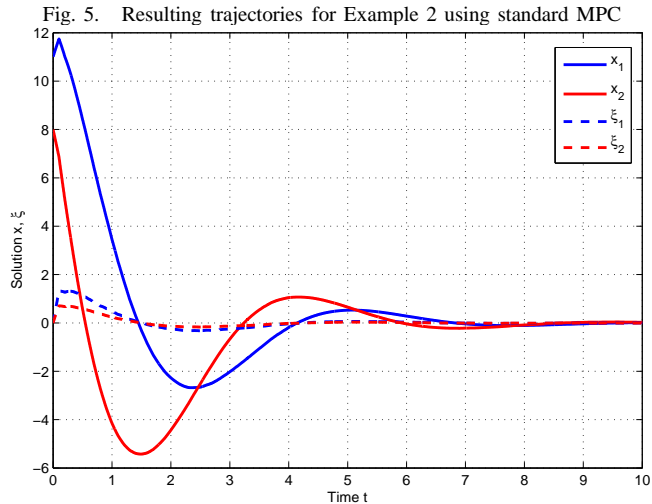}
\end{figure}
\begin{figure}[!ht]
	\caption{Resulting trajectories for Example 2 using OBPC}
	\label{Figure:OBPC:ex2}
	\includegraphics[width=0.48\textwidth]{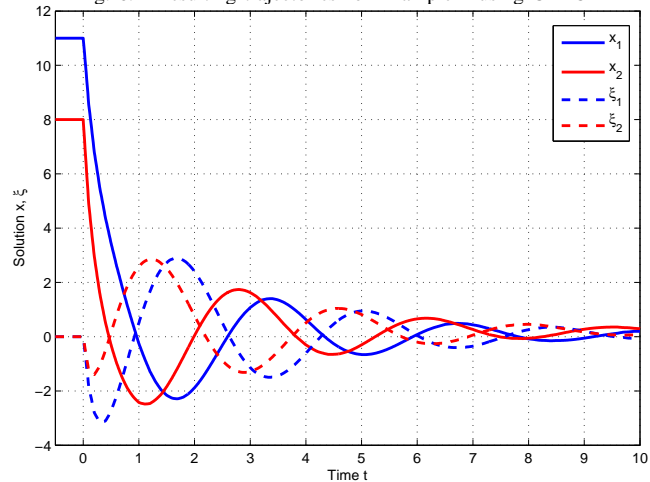}
\end{figure}
\FloatBarrier
This example shows the main difficulty arising in this context: The considered Luenberger observer uses a linear compensation for the deviation. The exponential evolution of this error due to the retarded nature of the observer, however, may exceed any linear bound if the Lipschitz constant of the system is greater then one. \\
The success of the presented setup will hence depend massively on the availability of retarded observers which show a tunable rate of convergence.
\section{Conclusion}\label{conclusion}
In this work we have shown the stability analysis of the newly proposed MPC scheme. Moreover we have shown that this approach may lead to a considerable improvement compared to the standard setup. Future research work concerns the detailed comparison analysis. Yet, we expect a better behavior of the solutions due to the more consistent computation of the control values. In particular the recovery time of the closed--loop system regarding disturbances will be a main issue. Another important aspect is development and analysis of retarded observers, at best with tunable rate of convergence.
\bibliography{bib_observer.bib}
\bibliographystyle{plain}
\end{document}